\documentclass[11pt]{article}
\usepackage{amsmath,amssymb,amsthm}
\usepackage{mathrsfs}
\usepackage{enumerate}
\usepackage{enumitem}
\usepackage{latexsym,amsmath,amssymb,amsfonts,epsfig,graphicx,cite,psfrag}
\usepackage{eepic,color,colordvi,amscd}
\usepackage{ebezier}
\usepackage{verbatim}

\newtheorem{thm}{Theorem}[section]
\newtheorem{lm}[thm]{Lemma}

\newtheorem{conj}{Conjecture}[section]

\newtheorem{cla}{Claim}

\usepackage{subfigure}

\begin{document}
	
\title{A stability result on matchings in 3-uniform hypergraphs}
\author{Mingyang Guo, {Hongliang Lu \footnote{Corresponding email: luhongliang215@sina.com (H. Lu)}, Dingjia Mao }\\
		\small School of Mathematics and Statistics\\
		Xi'an Jiaotong University \\ \small  Xi'an, Shaanxi 710049, P.R.China \\ }

\date{}

\maketitle

\begin{abstract}

Let $n,s,k$ be three positive integers such that $1\leq s\leq(n-k+1)/k$ and let $[n]=\{1,\ldots,n\}$.  Let $H$ be a   $k$-graph with vertex set $[n]$,  and let $e(H)$ denote the number of edges of $H$. Let $\nu(H)$ and $\tau(H)$ denote the size of a largest matching and the size of a minimum vertex cover in $H$, respectively.  Define $A^k_i(n,s):=\{e\in\binom{[n]}{k}:|e\cap[(s+1)i-1]|\geq i\}$ for $2\leq i\leq k$ and $HM^k_{n,s}:=\big\{e\in\binom{[n]}{k}:e\cap[s-1]\neq\emptyset\big\}
\cup\big\{S\big\}\cup \big\{e\in\binom{[n]}{k}: s\in e, e\cap S\neq \emptyset\}$, where $S=\{s+1,\ldots,s+k\}$. Frankl and Kupavskii proposed a conjecture that if $\nu(H)\leq s$ and $\tau(H)>s$, then $e(H)\leq \max\{|A^k_2(n,s)|,\ldots ,|A^k_k(n,s)|,|HM^k_{n,s}|\}$. In this paper, we prove this conjecture for $k=3$ and sufficiently large $n$.
\end{abstract}

\section{Introduction}

A \emph{hypergraph} $H$ is a pair $H=(V,E)$, where $V:=V(H)$ is a set of vertices and $E:=E(H)$ is a set of non-empty subsets of $V$. For a positive integer $k\geq 2$, a hypergraph is \emph{$k$-uniform} if $E\subseteq\binom{V}{k}$, where $\binom{V}{k}:=\{T\subseteq{V}: |T|=k\}$. A $k$-uniform hypergraph is also called a \emph{$k$-graph}. We use \emph{$l$-set} to denote a set of $l$ elements and for any integer $n$ define $[n]:=\{1,2,\cdots,n\}$. Throughout this paper, we often identify $E(H)$ with $H$ when there is no confusion.

Given a $k$-graph $H$ we write $e(H):=|E(H)|$. For a vertex $v\in V(H)$, let $N_H(v):=\{f\in\binom{V}{k-1}:f\cup \{v\}\in E(H)\}$. For any $T\subseteq V(H)$, we use $d_H(T)$ to denote the {\it degree}
of $T$ in $H$, i.e., the number of edges of $H$ containing $T$. We denote $d_H(\{v\})$ by $d_H(v)$.


A \emph{matching} in a hypergraph $H$ is a set of pairwise disjoint edges. The number of edges in a matching $M$ denoted by $|M|$ is called the \emph{size} of the matching. The size of the largest matching in $H$ is denoted by $\nu(H)$, known as the\emph{ matching number} of $H$. A matching is \emph{perfect} if it covers all vertices of $V(H)$. A \emph{vertex cover} in a hypergraph $H$ is a set of vertices which intersects all edges of $H$. We use $\tau(H)$ to denote the minimum size of a vertex cover in $H$. A set $I \subseteq V(H)$ that contains no edge of $H$ is called an
\emph{independent set}  in $H$. The size of a largest independent set in $H$ is denoted by $\alpha(H)$, known as the\emph{ independent number} of $H$.

Let $n,s,k$ be three positive integers such that $k\geq2$ and $n\geq ks+k-1$. For any $U\subseteq [n]$ with $|U|=k(s+1)-1$, define $D^k_{n,s}(U):=\{e\in \binom{[n]}{k}:e\subseteq U\}$. Let $U,W$ be a partition of $[n]$ such that $|W|=s$. Define $H^k_{n,s}(U,W):=\{e\in\binom{[n]}{k}: e\cap W\neq\emptyset\}$.
When there is no confusion, we denote  $H^k_{n,s}(U,W)$ and $D^k_{n,s}(U)$ by $H^k_{n,s}$ and $D^k_{n,s}$, respectively.
In 1965, Erd\H os \cite{E65} asked for the determination of the maximum possible number of edges that can appear in any $k$-graph $H$ with $\nu(H)\leq s$. He conjectured that $H^k_{n,s}$ and $D^k_{n,s}$ are the two extremal constructions of this problem.
\begin{conj}[Erd\H os Matching Conjecture \cite{E65}]\label{conj}
	Let $n,s,k$ be three positive integers such that $k\geq2$ and $n\geq k(s+1)-1$. If $H$ is a $k$-graph on $n$ vertices and $\nu(H)\leq s$, then
	\begin{equation*}
		e(H)\leq\max\left\{\binom{n}{k}-\binom{n-s}{k},\binom{k(s+1)-1}{k}\right\}.
	\end{equation*}
\end{conj}


There have been recent
activities on the Erd\H{o}s Matching Conjecture, see \cite{AFHRRS,BDE,EKR,EG,E65,F13,F,F17,FLM12,FK19,FK,HLS,LM,LYY,FRR}.  The Erd\H{o}s Matching Conjecture was verified by Erd\H os and Gallai \cite{EG} for $k=2$. For $k\geq 3$, it was proved by Bollob\'as, Daykin and Erd\H os \cite{BDE} for $n\geq2k^3s$.  Subsequently, Huang, Loh and Sudakov settled the conjecture for $n\geq3k^2s$. In 2013, Frankl \cite{F13} verified the conjecture for $n\geq(2s+1)k-s$. Currently the best range is $n\geq \frac{5}{3}sk-\frac{2}{3}s$ by Frankl and Kupavskii \cite{FK}.
As for the special case of $k=3$, Frankl, R\"odl and Ruci\'nski \cite{FRR} proved the conjecture for $n\geq 4s$.  In particular, the Erd\H{o}s Matching Conjecture  was settled  for $k = 3$ and sufficiently large $n$ in \cite{LM}, and finally, it was completely resolved  for $k=3$ in \cite{F17}.

\begin{thm}[\L uczak and Mieczkowska \cite{LM}]\label{erdos3}
	There exists an integer $n_0$ such that the following holds. Let $H$ be a $3$-graph on $n\geq n_0$ vertices and $s$ be an integer with $1\leq s\leq (n-2)/3$, if $\nu(H)\leq s$, then
	\begin{equation*}
		e(H)\leq \max\left\{\binom{n}{3}-\binom{n-s}{3},\binom{3s+2}{3}\right\}.	
	\end{equation*}
\end{thm}




Bollob\'as, Daykin and Erd\H os \cite{BDE} proved a stability result of Conjecture \ref{conj} for $n>2k^3s$.

\begin{thm}[Bollob\'as, Daykin and Erd\H os \cite{BDE}]\label{bollobas}
	Let $n,s,k$ be three positive integers such that $k\geq2$ and $n> 2k^3s$. If $H=(V,E)$ is a $k$-graph on $n$ vertices, $\nu(H)\leq s$ and
	\begin{equation}\label{hmkedge}
		e(H)>\binom{n}{k}-\binom{n-s}{k}-\binom{n-s-k}{k-1}+1,
	\end{equation}
	then $\tau(H)\leq s$.
\end{thm}

\noindent\textbf{Remark:}
The condition (\ref{hmkedge}) is tight. Define $HM^k_{n,s}:=\big\{e\in\binom{[n]}{k}:e\cap[s-1]\neq\emptyset\big\}
\cup\big\{S\big\}\cup\big\{e\in\binom{[n]}{k}:s\in e, e\cap S\neq\emptyset\big\}$, where $S:=\{s+1,\ldots,s+k\}$. Note that $e(HM^k_{n,s})
=\binom{n}{k}-\binom{n-s}{k}-\binom{n-s-k}{k-1}+1$, $\nu(HM^k_{n,s})=s$ and $\tau(HM^k_{n,s})=s+1$. Therefore, the condition (\ref{hmkedge}) is tight.


Hilton and Milner \cite{HM} proved that for a $k$-graph $H$, if $\nu(H)=1$ and $\tau(H)> 1$, then $e(H)\leq \binom{n-1}{k-1}-\binom{n-k-1}{k-1}+1$ for $n>2k$. Frankl and  Kupavskii \cite{FK19} proved that for $k\geq3$ and either $n\geq(s+\max\{25,2s+2\})k$ or $n\geq(2+o(1))sk$, where $o(1)$ is with respect to $s\rightarrow\infty$, if $H$ is a $k$-graph with $\nu(H)=s$ and $\tau(H)>s$, then $e(H)\leq\binom{n}{k}-\binom{n-s}{k}-\binom{n-s-k}{k-1}+1$. \par

Let $n,s,k$ be positive integers such that $k\geq2$ and $n
\geq ks+k-1$. For $2\leq i\leq k-1$, define $A^k_i(n,s):=\{e\in\binom{[n]}{k}:|e\cap[(s+1)i-1]|\geq i\}$.
Note that $\nu(A^k_i(n,s))=s$ and $\tau(A^k_i(n,s))> s$.
Frankl and Kupavskii \cite{FK19} proposed the following conjecture.
\begin{conj}[Frankl and Kupavskii \cite{FK19}]\label{conj1}
	Suppose that $H$ is a $k$-graph with set vertex $[n]$. If   $\nu(H)=s$ and $\tau(H)>s$, then
	\begin{equation*}
		e(H)\leq \max\{|A^k_2(n,s)|,\ldots,|A^k_{k-1}(n,s)|,|HM^k_{n,s}|,|D^k_{n,s}|\}.
	\end{equation*}
\end{conj}

We confirm Conjecture \ref{conj1} for $k=3$ and sufficiently large $n$.
\begin{thm}\label{main}
	There exists an integer $n_0$ such that the following holds. Let $H$ be a $3$-graph on $n\geq n_0$ vertices  and $s$ be an integer with  $n\geq 3s+2$. If $\nu(H)\leq s$ and $\tau(H)>s$, then
	\begin{equation*}\label{hm}
		e(H)\leq \max\left\{\binom{n}{3}-\binom{n-s}{3}-\binom{n-s-3}{2}+1,\binom{3s+2}{3}\right\}.
	\end{equation*}
\end{thm}

Given two $k$-graphs $H_1, H_2$ and a real number $\varepsilon>0$, we say that $H_2$ is \textit{$\varepsilon$-close} to $H_1$ if $V(H_1) = V(H_2)$ and $|E(H_1)\backslash E(H_2)|\leq\varepsilon|V(H_1)|^k$. Specially,  a $k$-graph $H$ on $n$ vertices is \textit{$\varepsilon$-close} to $D^k_{n,s}$ (or $H^k_{n,s}$) if  there is a partition $U,W$ of $V(H)$ with $|U|=3s+2$ (or $|W|=s$) such that $H$ is \textit{$\varepsilon$-close} to $D^k_{n,s}(U)$ ($H^k_{n,s}(U,W)$, respectively).
Given $0<\theta<1$, we say a vertex $v\in V(H)$ is \emph{$\theta$-good} with respect to $H'$ if $|N_{H'} (v)\setminus N_H(v)|\leq \theta n^{k-1}$. Otherwise we say that $v$ is \emph{$\theta$-bad}. 
For a  $k$-graph $H$ and  $S\subseteq V(H)$, we use $H-S$ to denote the hypergraph obtained from $H$ by deleting $S$ and all edges of $H$ intersecting set $S$, and we use $H[S]$ to denote the sub-hypergraph with vertex set $S$ and edge set $\{e\in E(H) : e\subseteq S\}$. For a  $k$-graph $H$ and  $E\subseteq E(H)$, we use $H-E$ to denote the hypergraph obtained from $H$ by deleting $E$. By $x\ll y$ we mean that for any $y>0$ there exists $x_0>0$ such that for any $x<x_0$ the following statement holds. We omit the floor and ceiling functions when they do not affect the proof.

The rest of the paper is organized as follows. In Section 2, we describe the so-called shifting method, which is a well-known technique in extremal set theory. The proof of Theorem \ref{main} will be divided into two parts depending on whether the graph is close to extremal graphs in Sections 3 and 4. In Section 3, we prove Theorem \ref{main} for the case $H$ is close to $H^3_{n,s}$ or $D^3_{n,s}$. For the
 case that $H$ is not close to  $H^3_{n,s}$ or $D^3_{n,s}$, we complete the proof with three steps in Section 4: firstly, we construct a 3-graph $H'$ such that $H$ has a matching of size $s$ if and only if $H'$ has an almost perfect matching;  secondly, we use  a recent approach of Han-Kohayakawa-Person \cite{HKP18} and Han \cite{Han19} to find  edge-disjoint fractional perfect matchings in $H'$; finally, we use   the round randomization method of Alon, Frankl, Huang, R\"odl,
Ruci\'nski, and Sudakov \cite{AFHRRS} to convert fractional perfect matchings into an integral almost perfect matching.


\section{shifting}
Let $H$ be a $k$-graph on vertex set $[n]$. For vertices $1\leq i<j\leq n$, we define the \emph{$(i,j)$-shift} $S_{ij}$ by $S_{ij}(H)=\{S_{ij}(e):e\in E(H)\}$, where
$$S_{ij}(e)=\begin{cases}
	e\setminus \{j\}\cup \{i\}, & \text{if $j\in e$, $i\notin e$ and $e\setminus \{j\}\cup \{i\}\notin E(H)$;}\\
	e, &\text{otherwise.}
\end{cases}$$
The following well-known result can be found in \cite{F95}.
\begin{lm}\label{shiftprop}
	The $(i,j)$-shift satisfies the following properties.
	\begin{enumerate}[itemsep=0pt,parsep=0pt,label=$($\roman*$)$]
				\item $e(H)=e(S_{ij}(H))$ and $|e|=|S_{ij}(e)|$,
				\item $\nu(S_{ij}(H))\leq \nu(H)$.
	\end{enumerate}
\end{lm}
A $k$-graph $H$ is called \emph{stable} if $H=S_{ij}(H)$ for all $1\leq i<j\leq n$. It is not difficult to see that if $H$ is a stable $k$-graph, then for any subsets $\{u_1,\ldots,u_k\}, \{v_1,\ldots,v_k\}\subset[n]$ such that $u_i\leq v_i$ for $i\in[k]$, $\{v_1,\ldots,v_k\}\in E(H)$ implies $\{u_1,\ldots,u_k\}\in E(H)$.

\section{$\varepsilon$-close case}
When $H$ is $\varepsilon$-close to $D^k_{n,s}$, the proof of Lemma 2 in \cite{LM} implies the following theorem.

\begin{thm}[\L uczak and Mieczkowska \cite{LM}]\label{closeclique}
	For any given integer $k\geq 3$, there exist $\varepsilon>0$ and a positive integer $n_0$ such that the following holds. Let $H$ be a $k$-graph on $n>n_0$ vertices and $s$ be an integer with $n(1/k-1/(2k^2))\leq s\leq (n-k+1)/k$. If $H$ has a complete subgraph of size at least $(1-\varepsilon)ks$ and $\nu(H)\leq s$, then
	\begin{equation*}
		e(H)\leq\binom{ks+k-1}{k}.
	\end{equation*}
\end{thm}

To deal with the case that $H$ is $\varepsilon$-close to $D^k_{n,s}$, we also need the following lemma.
\begin{lm}\label{closecli2}
	Let $k\geq 3$ be an integer and $\varepsilon,c$ be reals  such that $0< \varepsilon\ll c\ll 1/k$. Let $n,s$ be integers such that $n/2k^3\leq s\leq (1-c)n/k$. Let $H$ be a $k$-graph on $n$ vertices and $U$ be a subset of $V(H)$ of size $(ks+k-1)$. If $H$ is $\varepsilon$-close to $D^k_{n,s}(U)$ and there are at least $\varepsilon^{\frac{1}{2k}} n/(k+1)$ vertices in $V(H)\setminus U$ with degree at least $\varepsilon^{\frac{1}{2k}} n^{k-1}$, then $\nu(H)> s$.
\end{lm}
\begin{proof}
Since $H$ is $\varepsilon$-close to $D^k_{n,s}(U)$, one can see that $|E(D^k_{n,s}(U))\setminus E(H)|\leq \varepsilon n^k$. Then all but at most $k\sqrt{\varepsilon}n$ vertices in $H$ are $\sqrt{\varepsilon}$-good. Otherwise,
\begin{equation*}
	\begin{split}
		|E(D^k_{n,s}(U))\setminus E(H)|&=\frac{1}{k}\sum_{v\in V(H)}|N_{D^k_{n,s}(U)}(v)\backslash N_{H}(v)|\\
		&>(k\sqrt{\varepsilon}n\cdot \sqrt{\varepsilon}n^{k-1})/k=\varepsilon n^k,
	\end{split}
\end{equation*}
a contradiction.

By assumption there are at least $\varepsilon^{\frac{1}{2k}} n/(k+1)$ vertices in $V(H)\setminus U$ with degree at least $\varepsilon^{\frac{1}{2k}} n^{k-1}$ in $H$. 
So we choose  $S\subseteq  V(H)\setminus U$ such that $|S|=\varepsilon^{\frac{1}{2k}} n/(k+1)$ and $d_{H}(v)\geq \varepsilon^{\frac{1}{2k}} n^{k-1}$ for every $v\in S$.

Let $M_1$ be a maximum matching in $H$ such that $|e\cap S|=1$ for every $e\in M_1$. We claim that $|M_1|=|S|$, otherwise, suppose that $|M_1|<|S|= \varepsilon^{\frac{1}{2k}} n/(k+1)$, then there exists a vertex $v\in S\setminus V(M_1)$ such that $N_{H}(v)\cap \binom{V(H)\setminus (V(M_1)\cup S)}{k-1}=\emptyset$. Thus $d_H(v)\leq |V(M_1)\cup S|\cdot n^{k-2}\leq k\varepsilon^{\frac{1}{2k}} n^{k-1}/(k+1)$, contradicting  the fact that $d_H(v)\geq \varepsilon^{\frac{1}{2k}} n^{k-1}$.

 Recall that all but at most $k\sqrt{\varepsilon}n$ vertices in $H$ are $\sqrt{\varepsilon}$-good. So  there exists a subset $S'$ of $\sqrt{\varepsilon}$-good vertices in $H$ such that $S'\subseteq U\setminus V(M_1)$ and $|S'|=s-\varepsilon^{\frac{1}{2k}} n/(k+1)+1$. Let $G:=H-V(M_1)$. Let $M_2$ be a maximum matching in $G$ such that $|e\cap S'|=1$ for every $e\in M_2$. We claim that $|M_2|\geq s-\varepsilon^{\frac{1}{2k}} n/(k+1)+1$. Otherwise, suppose that $|M_2|\leq s-\varepsilon^{\frac{1}{2k}} n/(k+1)$, then there exists a vertex $v\in S'\setminus V(M_2)$ such that
$$N_{G}(v)\cap \binom{U\setminus (V(M_1)\cup V(M_2)\cup S')}{k-1}=\emptyset.$$
Note that $|U\cap V(M_1)|\leq (k-1)\varepsilon^{\frac{1}{2k}} n/(k+1)$. So we have
\[
|U\setminus (V(M_1)\cup V(M_2)\cup S')|> ks-(k-1)\varepsilon^{\frac{1}{2k}} n/(k+1)-k(s-\varepsilon^{\frac{1}{2k}} n/(k+1))\geq\varepsilon^{\frac{1}{2k}}n/(k+1).
\]
 Thus
\begin{equation*}
	|N_{D^k_{n,s}(U)}(v)\setminus N_{H}(v)|\geq\binom{|U\setminus (V(M_1)\cup V(M_2)\cup S')|}{k-1}>\binom{\varepsilon^{\frac{1}{2k}}n/(k+1)}{k-1}> \sqrt{\varepsilon}n^{k-1},
\end{equation*}
contradicting the fact that $v$ is $\sqrt{\varepsilon}$-good in $H$. Then we can construct a matching $M_2$ of size $s-\varepsilon^{\frac{1}{2k}} n/(k+1)+1$ such that every edge in $M_2$ contains exactly one vertex belonging to $S'$. $M_1\cup M_2$ is a matching of size $s+1$ in $H$.
\end{proof}

When $H$ is close to $H^k_{n,s}$, we prove the following lemma.
\begin{lm}\label{close}
	
	For any given integer $k\geq 3$, there exist $\varepsilon>0$ and a positive integer $n_0$ such that the following holds. Let $H$ be a $k$-graph on $n>n_0$ vertices and $s$ be an integer with $1\leq s\leq n(1/k-2/(5k^2))$. If $H$ is $\varepsilon$-close to  $H^k_{n,s}$, $\nu(H)\leq s$ and
	\begin{equation}\label{hmedge}
		e(H)>\binom{n}{k}-\binom{n-s}{k}-\binom{n-s-k}{k-1}+1,
	\end{equation}
	then $\tau(H)\leq s$.
\end{lm}

\begin{proof}
	Since $H$ is $\varepsilon$-close to $H^k_{n,s}$, there is a partition $U,W$ of $V(H)$ such that $|U|=n-s$, $|W|=s$ and $|E(H^k_{n,s}(U,W)\backslash E(H)|\leq\varepsilon n^k$. Let $T\subset W$ be the set of vertices which are contained in at most $\binom{n-1}{k-1}-\binom{n-ks-1}{k-1}$ edges of $H$ and let $t=|T|$. Let $W':=W\setminus T$ and let $H':=H-W'$.
	\begin{cla}\label{claim1}
		$\nu(H')\leq t$.
	\end{cla}
	
	Suppose that $\nu(H')> t$. Let $M_1$ be a   matching of size $t+1$ in $H'$. Let $\{v_1,\ldots,v_{s-t}\}:=W'$.
Next we  greedily construct a matching $M_2$ of size $s-t$ in
$H-V(M_1)$ such that $|e\cap W'|=1$ for all $e\in M_2$.
 For $v_1\in W'$, note that
 $d_H(v_1)>\binom{n-1}{k-1}-\binom{n-ks-1}{k-1}\geq \binom{n-1}{k-1}-\binom{n-k(t+1)-(s-t-1)-1}{k-1}$,
so there exists an edge $e_1\in E(H-V(M_1))$ such that
$e_1\cap W'=\{v_1\}$.
Now suppose we have found a matching
$\{e_1,e_2,...,e_r\}$ in $H-V(M_1)$  such that
$|e_i\cap W'|=\{v_i\}$ for all $i\in [r]$.
If $r=s-t$, then $M_1\cup \{e_1,\ldots, e_{s-t}\}$ is a desired matching.
So we may assume that $r<s-t$.
Write $G_r:=H-V(M_1)-(\cup_{i=1}^r e_i)$. Note that $|[n]\setminus (W'\cup V(M_1)\cup(\cup_{i=1}^r e_i))|>  0.4n/k-k$.
Since $d_H(v_{r+1})>\binom{n-1}{k-1}-\binom{n-ks-1}{k-1}\geq \binom{n-1}{k-1}-\binom{n-k(t+r+1)-(s-t-r-1)-1}{k-1}$,  there exists an edge $e_{r+1}\in E(G_r)$ such that
$e_{r+1}\cap W'=\{v_{r+1}\}$. Continuing the process, we may find the desired matching $M_2=\{e_1,\ldots,e_{s-t}\}$.
 Now $M_1\cup M_2$ is a matching of size $s+1$ in $H$, a contradiction.

%
	
	\begin{cla}\label{claim2}
		$t\leq(n-s+t)/2k^3$.
	\end{cla}
	
	Since $s\leq n(1/k-2/(5k^2))$, each vertex $v\in T$ is contained in at most
	\begin{equation}\label{2}
		\binom{n-1}{k-1}-\binom{n-ks-1}{k-1}\leq \left(1-\left(\frac{1}{5k}\right)^{k-1}\right)\binom{n-1}{k-1}
	\end{equation}
	edges. Thus,
	\begin{equation}\label{1}
		|E(H^k_{n,s}(U,W))\backslash E(H)|\geq \frac{1}{k}\sum_{v\in T}|N_{H^k_{n,s}(U,W)}(v)\setminus N_{H}(v)|\geq\frac{t}{k}\left(\frac{1}{5k}\right)^{k-1}\binom{n-1}{k-1}.
	\end{equation}
	Furthermore, since $H$ is $\varepsilon$-close to $H^k_{n,s}(U,W)$, we have
	\begin{equation*}
		|E(H^k_{n,s}(U,W))\backslash E(H)|\leq \varepsilon n^k.
	\end{equation*}
	Compared with inequality (\ref{1}), we have
$t\leq \frac{5^{k-1}k^k\varepsilon n^k}{\binom{n-1}{k-1}}$.
Since $\varepsilon$ is small enough and $n$ is sufficiently large, then $t\leq n/3k^3\leq(n-s+t)/2k^3 $. This completes the proof of Claim \ref{claim2}.

	By inequality (\ref{hmedge}), we have
	\begin{equation*}
		\begin{split}
			e(H')&\geq e(H)-\left(\binom{n}{k}-\binom{n-s+t}{k}\right)\\
			&>\binom{n}{k}-\binom{n-s}{k}-\binom{n-s-k}{k-1}+1-\binom{n}{k}+\binom{n-s+t}{k}\\
			&=\binom{n-s+t}{k}-\binom{(n-s+t)-t}{k}-\binom{(n-s+t)-t-k}{k-1}+1.
		\end{split}
	\end{equation*}
	
	Now Theorem \ref{bollobas} implies that $\tau(H')\leq t$. Consequently, one can see that $\tau(H)\leq s$.
\end{proof}


\begin{lm}\label{closecase}
	There exist $\varepsilon>0$ and a positive integer $n_0$ such that the following holds. Let $H$ be a $3$-graph on $n>n_0$ vertices and $s$ be an integer with $n/54\leq s\leq 13n/45$. If $H$ is $\varepsilon$-close to $H^3_{n,s}$ or $D^3_{n,s}$, $\nu(H)\leq s$  and $\tau(H)>s$, then
	\begin{equation}
		e(H)\leq\max\left\{\binom{n}{3}-\binom{n-s}{3}-\binom{n-s-3}{2}+1,\binom{3s+2}{3}\right\}.
	\end{equation}
\end{lm}
\begin{proof}
By Lemma \ref{close}, we may assume that $H$ is $\varepsilon$-close to $D^3_{n,s}$. So there exists a subset $U\subseteq V(H)$ of size $3s+2$ such that $|E(D^3_{n,s}(U))\setminus E(H)|\leq \varepsilon n^3$.

For $5n/18< s\leq 13n/45$, let $U=[3s+2]$ and $V(H)\setminus U=[n]\setminus [3s+2]$. Iterating the $(i,j)$-shift for all $1\leq i<j\leq n$ will eventually produce a $3$-graph $H'$ which is invariant with respect to all $(i,j)$-shifts. By Lemma \ref{shiftprop}, we have $e(H')=e(H)$ and $\nu(H')\leq \nu(H)\leq s$. By the definition of $(i,j)$-shift, $|E(D^3_{n,s}(U))\setminus E(H')|\leq \varepsilon n^3$. We claim that there is a complete subgraph of size at least $3(1-3\varepsilon^{\frac{1}{3}})s$ in $H'$. Let $U'=[3(1-3\varepsilon^{\frac{1}{3}})s]$. Suppose that $H'[U']$ is not a complete subgraph, then $U\setminus U'$ is an independent set in $H'$ since $H'$ is stable. Thus $|E(D^3_{n,s}(U))\setminus E(H')|\geq \binom{|U\setminus U'|}{3}>\binom{5\varepsilon^{\frac{1}{3}}n/2}{3}> \varepsilon n^3$ for sufficiently large $n$, which is a contradiction. So by Theorem \ref{closeclique}, $e(H)= e(H')\leq \binom{3s+2}{3}$.

For $n/54\leq s\leq 5n/18$. Let $s=\alpha n$, then $1/54\leq \alpha \leq 5/18$. Let $f(x)=\frac{1-(1-x)^3}{6}-\frac{9x^3}{2}$, we have
	\begin{equation}\label{diff}
	\begin{split}
		\binom{n}{3}-\binom{n-s}{3}-\binom{n-s-3}{2}+1-\binom{3s+2}{3}&=\frac{(1-(1-\alpha)^3)n^3}{6}-\frac{9\alpha^3n^3}{2}+o(n^{3})\\&=f(\alpha)n^3+o(n^3).
	\end{split}
\end{equation}
Since $f'(x)=\frac{1-2x-26x^2}{2}$ is decreasing in $[1/54,5/18]$ with $f'(1/54)>0$ and $f'(5/18)<0$, we have $f(\alpha)\geq \min\{f(1/54),f(5/18)\}=f(5/18)>0.007$ for $1/54\leq \alpha \leq 5/18$.

Since $\nu(H)\leq s$, Lemma \ref{closecli2} implies that there are at most $\varepsilon^{1/6}n/4$ vertices with degree at least $\varepsilon^{1/6}n^2$. Thus the number of edges intersecting $V(H)\setminus U$ is no more than
\[
\varepsilon^{1/6}n^3/4+\varepsilon^{1/6}n^3= 5\varepsilon^{1/6}n^3/4.
\]
Thus we have
$$e(H)\leq \binom{3s+2}{3}+5\varepsilon^{1/6}n^3/4.$$
Compared with inequality (\ref{diff}), we have
$$e(H)\leq \binom{n}{3}-\binom{n-s}{3}-\binom{n-s-3}{2}+1$$
for sufficiently small $\varepsilon$.
\end{proof}

\section{non-close case}

A \emph{fractional matching} of a hypergraph $H$ is a function $f:E(H)\rightarrow [0,1]$ such that for each $v\in V(H)$, $\sum_{v\in e}f(e)\leq 1$. A fractional matching is called {\it fractional perfect matching} if $\sum_{e\in E}f(e)=|V(H)|/k$, or equivalently, $\sum_{v\in e}f(e)=1$ for all $v\in V(H)$. Let
\begin{equation*}
	\nu^*(H)=\max\left\{\sum_{e\in E}f(e):f \mbox{ is a fractional matching of } H\right\}.
\end{equation*}
A \emph{fractional vertex cover} of a hypergraph $H$ is a function $w:V(H)\rightarrow [0,1]$ such that for each $e\in E$, $\sum_{v\in e}w(v)\geq 1$. Let
\begin{equation*}
	\tau^*(H)=\min\left\{\sum_{v\in V}w(v):w \mbox{ is a fractional vertex cover of } H\right\}.
\end{equation*}
Then the strong duality theorem of linear programming gives
\begin{equation*}
	\nu^*(H)=\tau^*(H).
\end{equation*}

In this section we study the case that $H$ is close to neither $H^3_{n,s}$ nor $D^3_{n,s}$. In order to complete the proof of Theorem \ref{main} for $n/54\leq s\leq 13n/45$, it is sufficient for us to prove the following lemma.
\begin{lm}\label{nonclose}
For every $0<\varepsilon\ll 1$ there exists a positive integer $n_0$ such that the following holds. Let $n,s$ be integers with $n>n_0$ and $n/54\leq s\leq 13n/45$. Let $H$ be a $3$-graph on vertex set $[n]$. If $H$ is $\varepsilon$-close to neither $H^3_{n,s}$ nor $D^3_{n,s}$ and
\begin{equation}\label{nonclose-ineq}
e(H)>\max\left\{\binom{n}{3}-\binom{n-s}{3}-\binom{n-s-3}{2}+1,\binom{3s+2}{3}\right\},
\end{equation}
then $\nu(H)> s$.

\end{lm}


For stable $3$-graphs, the proof of Lemma 7 in \cite{LM} and Lemma 4.2 in \cite{GLMY} imply the following theorem.
\begin{thm}[\L uczak and Mieczkowska \cite{LM}; Gao et.al., \cite{GLMY}]\label{shift}
	Let $\varepsilon,\rho$ be two reals such that $0<\rho\ll\varepsilon<1$. Let $n,s$ be two positive integers such that $n$ is sufficiently large and $1\leq s\leq (n-2)/3$. Let $H$ be a $3$-graph with $V(H)=[n]$ such that $H$ is stable. If
	\begin{equation*}
		e(H)>\max\left\{\binom{n}{3}-\binom{n-s}{3},\binom{3s+2}{3}\right\}-\rho n^3
	\end{equation*}
	and $\nu(H)\leq s$, then $H$ is $\varepsilon$-close to $H^3_{n,s}([n]\setminus[s],[s])$ or $D^3_{n,s}([3s+2])$.
\end{thm}


\begin{lm}[K\"uhn, Osthus and Treglown \cite{K13}]\label{vexclose}
	Let $0<\theta<10^{-6}$ and let $n,s$ be two positive integers such that $n/150\leq s\leq n/3$. Let $H$ be a $3$-graph on $n$ vertices and $U,W$ be a partition of $V(H)$ such that $|W|=s$. If every vertex of $H$ is $\theta$-good with respect to $H^3_{n,s}(U,W)$. Then $\nu(H)\geq s$.
\end{lm}

In order to prove Lemma \ref{nonclose}, we need the following lemma.
\begin{lm}\label{frac}
Let $\theta,\eta,\rho,\varepsilon$ be reals such that $0<\theta\ll\eta\ll\rho\ll\varepsilon\ll 1$.  Let $n,s$ be two positive integers such that $n$ is sufficiently large and $n/54\leq s\leq 13n/45$. Let $r$ be a positive integer such that $2r\geq n-3(s+\eta n)$ and $n+r\equiv 0 \pmod 3$. Let $H$ be a $3$-graph on vertex set $[n+r]$. If $H[[n]]$ is not $\varepsilon$-close to  $H^3_{n,s}$ or $D^3_{n,s}$,   $e(H[[n]])\geq\max\{\binom{n}{3}-\binom{n-s}{3},\binom{3s+2}{3}\}-\rho n^3$ and every vertex in $V(H)$ is $\theta$-good with respect to $H^3_{n+r,r}([n],[n+r]\setminus [n])$, then $H$ has a fractional perfect  matching.
\end{lm}


\begin{proof}
	Let $\rho',\beta$ be two constants such that $\rho\ll\rho'\ll\beta\ll\varepsilon$.  Let $\omega:V(H)\rightarrow [0,1]$ be a minimum fractional vertex cover of  $H$. Rename the vertices in $[n]$ such that $\omega(1)\geq \cdots\geq \omega(n)$. Let $H'$ be a 3-graph with vertex set $V(H)$ and edge set $E(H')$, where
	\begin{align*}
		E(H'):=\Big\{e\in {V(H)\choose 3}\ |\ \sum_{v\in e}\omega(v)\geq 1\Big\}.
	\end{align*}
	One can see that $\omega$ is also a fractional vertex cover of $H'$ and $H$ is a subgraph of $H'$. Thus the size of minimum fractional cover of $H$ is no more than that of $H'$. So $\omega$ is also a minimum fractional vertex cover of $H'$. By Linear Programming Duality Theory, we have
	$\nu^*(H)=\tau^*(H)=\tau^*(H')=\nu^*(H')$. Thus it is sufficient for us to show that $H'$ has a fractional perfect matching. Next we will show  $H'$ has a perfect matching.

	By the definition of $H'$, $H'[[n]]$ is stable. Let $G:=H'[[n]]$ and $s':=s+\eta n$.

	\medskip
	\textbf{Claim 1.} $\nu(G)\geq s'$.
	
	Note that $e(G)\geq e(H[[n]])$. Thus
	\begin{equation}\label{5}
		\begin{split}
			e(G)>&\max\left\{\binom{n}{3}-\binom{n-s}{3},\binom{3s+2}{3}\right\}-\rho n^3\\
			\geq & \max\left\{\binom{n}{3}-\binom{n-s'}{3},\binom{3s'+2}{3}\right\}-\rho'n^3.
		\end{split}
	\end{equation}
If $G$ is not $\beta$-close to $D^3_{n,s'}$ or $H^3_{n,s'}$, by Theorem \ref{shift} and inequality (\ref{5}), we have $\nu(G)\geq s'$ since $G$ is stable. So we may assume that $G$ is $\beta$-close
$D^3_{n,s'}$ or $H^3_{n,s'}$.
 %
	
	Firstly, we consider that $G$ is $\beta$-close to $H^3_{n,s'}([n]\setminus[s'],[s'])$.
	Then all but at most $3\sqrt{\beta}n$ vertices in $G$ are $\sqrt{\beta}$-good. Otherwise,
	\begin{equation*}
		\begin{split}
			|E(H^3_{n,s'}([n]\setminus[s'],[s']))\backslash E(G)|&=\frac{1}{3}\sum_{v\in V(H)}|N_{H^3_{n,s'}([n]\setminus[s'],[s'])}(v)\backslash N_{G}(v)|\\
			&>(3\sqrt{\beta}n\cdot \sqrt{\beta}n^2)/3=\beta n^3,
		\end{split}
	\end{equation*}
a contradiction.

	Let $U:=[n]\setminus[s']$ and $W:=[s']$. Let $V^{bad}$ be the set of $\sqrt{\beta}$-bad vertices in $V(G)$. So $|V^{bad}|\leq3\sqrt{\beta}n$. Write $W^{bad}:=V^{bad}\cap W$ and $U^{bad}:=V^{bad}\setminus W^{bad}$. Let $b:=|W^{bad}|$ and $a:=|U^{bad}|$.
	Since $H[[n]]$ is not $\varepsilon$-close to $H^3_{n,s}$, we have $\alpha(H[[n]])< n-s-\varepsilon n/2 $.
	Otherwise, there exists an independent set $I_1$ of size $n-s-\varepsilon n/2 $. Let $W'$ be a subset of $[n]\setminus I_1$ such that $|W'|=s$. Then
	\begin{equation*}
		\begin{split}
			&|E(H^3_{n,s}([n]\setminus W',W'))\setminus E(H)|\\
			<&\binom{n}{3}-\binom{n-s}{3}-\left(\binom{n}{3}-\binom{n-s}{3}-\rho n^3-\varepsilon n^3/2 \right)\\
			<&\varepsilon n^3,
		\end{split}
	\end{equation*}
	contradicting that  $H[[n]]$ is not $\varepsilon$-close to $H^3_{n,s}$. Since $G$ is stable, we may assume that $I:=\{n-\alpha(G)+1,\ldots,n\}$ is a maximum independent set in $G$.  Let $S:=U\setminus I$.
	Since $\alpha(G)\leq \alpha(H[[n]])< n-s-\varepsilon n/2$, we have $|S|\geq\varepsilon n/3$. For any three distinct vertices $x_1,x_2,x_3\in S$, one can see that $\{x_1,x_2,x_3\}\in E(G)$. Otherwise $\{x_3\}\cup I$ is an independent set since $G$ is stable,
	contradicting that $I$ is a maximum independent set. Since $|S|\geq\varepsilon n/3\geq9\sqrt{\beta}n\geq 3b$, there exists a matching $M_1$ of size $b$ in $G[S]$. Let $G':=G-V(M_1)-V^{bad}$ and $U_1:=U\setminus (V(M_1)\cup U^{bad})$, $W_1:=W\backslash  W^{bad}$. Recall that every vertex in $U_1\cup W_1$ is $\sqrt{\beta}$-good in $G$ with respect to $H^3_{n,s'}(U,W)$. Thus for $x\in V(G')$,
	\begin{equation*}
			|N_{H^3_{|V(G')|,|W_1|}(U_1,W_1)}(x)\setminus N_{G'}(x)|\leq|N_{H^3_{n,s'}(U,W)}(x)\backslash N_{G}(x)|
			\leq \sqrt{\beta}n^2\leq2\sqrt{\beta}|V(G')|^2.
	\end{equation*}
	So every vertex in $V(G')$ is $2\sqrt{\beta}$-good with respect to $H^3_{|V(G)'|,|W_1|}(U_1,W_1)$.
	By Lemma \ref{vexclose}, $G'$ has a matching $M_2$ of size $s'-b$. Then $M_1\cup M_2$ is a matching of size $s'$ in $G$.

Secondly, consider that  $G$ is $\beta$-close to $D^3_{n,s'}([3s'+2])$.
	Let $U'$ be a subset of $[3s'+2]$ such that $|U'|=3s+2$.
	Since $H[[n]]$ is not $\varepsilon$-close to $D^3_{n,s}$, there are at least $\varepsilon n/3$ vertices in $[n]\setminus U'$ with degree at least $\varepsilon n^2/2$ in $H[[n]]$. Otherwise, the number of edges intersecting $[n]\setminus U'$ is at most
	$$\frac{1}{2}\varepsilon n^2 \cdot n+\frac{1}{3}\varepsilon n\cdot {n\choose 2}\leq\frac{2}{3}\varepsilon n^3;$$
	and so we have
	\begin{equation*}
		|E(D^3_{n,s}(U'))\setminus E(H[[n]])|<\binom{3s+2}{3}-\left(e(H)-\frac{2}{3}\varepsilon n^3\right)<\varepsilon n^3,
	\end{equation*}
	a contradiction. Since $E(H[[n]])\subseteq E(G)$, there are at least $\varepsilon n/3-|[3s'+2]\setminus U'|$ vertices in $[n]\setminus [3s'+2]$ with degree at least $\varepsilon n^2/2$ in $G$. Since $\varepsilon n/3-|[3s'+2]\setminus U'|\geq\beta^{1/6} n/4$ and $\varepsilon n^2/2\geq \beta^{1/6} n^2$, by Lemma \ref{closecli2}, we have $\nu(G)>s'$. This completes the proof of Claim 1.

	Let $M$ be a matching of size $s+\eta n$ in $G$. Since $n+r\equiv 0 \pmod 3$, we have $n+r-|V(M)|\equiv 0 \pmod 3$. Since $2r\geq n-3\eta n-3s>n/15$, there is a subset $W_2\subseteq \{n+1,\ldots,n+r\}$ of size $(n+r-|V(M)|)/3>n/45+r/3$. Let $H'':=H'-V(M)$ and $U_2:=V(H)\setminus(W_2\cup V(M))$. Note that $|U_2|=2|W_2|$ and $|V(H'')|=|U_2|+|W_2|\geq n/15+r$. Since every vertex in $V(H)$ is $\theta$-good with respect to $H^3_{n+r,r}([n],[n+r]\setminus [n])$, one can see that for every $x\in V(H'')$,
	\begin{equation*}
		\begin{split}
			|N_{H^3_{|V(H'')|,|W_2|}(U_2,W_2)}(x)\setminus N_{H''}(x)|&\leq|N_{H^3_{n+r,r}([n],[n+r]\setminus [n])}(x)\backslash N_{H}(x)| \\
			&\leq \theta (n+r)^2\leq\sqrt{\theta}|V(H'')|^2.
		\end{split}
	\end{equation*}
That is, every vertex in $V(H'')$ is $\sqrt{\theta}$-good with respect to $H^3_{|V(H'')|,|W_2|}(U_2,W_2)$. By Lemma \ref{vexclose}, there is a matching $M'$ of size $|W_2|$ in $H''$. Note that
\[
|M\cup M'|=s+\eta n+(n+r-|V(M)|)/3=s+\eta n+(n+r-3s-3\eta n)/3=(n+r)/3.
\]
 $M\cup M'$ is a perfect matching in $H'$. So $\nu^*(H')=\nu(H')=(n+r)/3$. By linear programming duality theorem, $H$ has a fractional perfect  matching. This completes the proof.
\end{proof}

For a given $k$-graph $H$, let $V'$ be a set of $r$ vertices such that $V(H)\cap V'=\emptyset$.
Define $H_r^k$ to be a $k$-graph with vertex set $V(H)\cup V'$ and edge set
$$E(H_r^k)=E(H)\cup \left\{e\in {V(H)\cup V'\choose k}\ :\ e\cap V'\neq \emptyset\right\}.$$

To prove Lemma \ref{nonclose}, we will find an almost perfect matching of size at least $r+s+1$ in $H_r^3$ which implies that there exists a matching of size $s+1$ in $H$. To find such an almost perfect matching, we will find
an almost regular subgraph of $H$ with bounded maximum $2$-degree  by using  Lemma \ref{frac} and round randomization method.  Then we apply  the following theorem of Frankl and R\"{o}dl \cite{F85}. For any positive integer $l$, we use $\Delta_l(H)$ to denote the maximum $l$-degree of a hypergraph $H$.

\begin{thm}[Frankl and R\"{o}dl \cite{F85}]\label{deg}
	For every integer $k\geq 2$ and a real number $\varepsilon >0$, there exists $\tau=\tau(k,\varepsilon)$, $d_0=d_0(k,\varepsilon)$ such that for every $n\geq D\geq d_0$ the following holds: Every $k$-graph $H$ on $n$ vertices with $(1-\tau)D< d_H(v)<(1+\tau)D$ for all $v\in V(H)$ and $\Delta_2(H)<\tau D$ contains a matching covering all but at most $\varepsilon n$ vertices.
\end{thm}


The following well-known  Chernoff bounds can be found  in \cite{J00} (see Theorem 2.8, inequalities (2.9) and (2.11)). We denote by $Bi(n,p)$ a binomial random variable with parameters $n$ and $p$.

\begin{lm}[Chernoff inequality for small deviation]\label{bound}
	If $X=\sum_{i=1}^nX_i$, each random variable $X_i$ has Bernoulli distribution with expectation $p_i$, and $\alpha\leq3/2$, then
	\begin{equation}\label{chers}
		\mathbb{P}(|X-\mathbb{E}X|\geq\alpha\mathbb{E}X)\leq2e^{-\frac{\alpha^2}{3}\mathbb{E}X}.
	\end{equation}
	In particular, when $X\thicksim Bi(n,p)$ and $\lambda <\frac{3}{2}np$, then
	\begin{equation*}
		\mathbb{P}(|X-np|\geq\lambda)\leq e^{-\Omega(\lambda^2/(np))}.
	\end{equation*}
\end{lm}

\begin{lm}[Chernoff inequality for large deviation]
	If $X=\sum_{i=1}^nX_i$, each random variable $X_i$ has Bernoulli distribution with expectation $p_i$, and $x\geq7\mathbb{E}X$, then
	\begin{equation}\label{cherl}
		\mathbb{P}(X\geq x)\leq e^{-x}.
	\end{equation}
\end{lm}

\noindent \textbf{Proof of Lemma \ref{nonclose}.}
Let  $\eta,\theta,\rho$ be reals such that $0<\theta\ll\eta\ll \rho\ll\varepsilon\ll1$. We choose a positive integer $r$ such that $n-3s-2\eta n\leq 2r\leq n-3s-\eta n$ and $n+r\equiv 0 \pmod 3$. Let $Q:=\{n+1,\ldots,n+r\}$ and $n_1:=n+r$. Recall that $H_r^3$ is a $3$-graph with vertex set $[n_1]$ and edge set $E(H_r^3)=E(H)\cup \{e\in {[n_1]\choose 3}\ :\ e\cap Q\neq \emptyset\}$. For proving $\nu(H)> s$, it is sufficient for us to show that $\nu(H_r^3)> s+r$. Indeed, let $M$ be a matching of size $s+r+1$ in $H^3_r$, then there are at most $r$ edges in $M$ intersecting $Q$.

We are going to find $n^{0.2}$ fractional perfect matchings $f_1,\ldots,f_{n^{0.2}}$ in $H_r^3$ such that
$$\sum_{\{x,y\}\subseteq e}\sum_{i=1}^{n^{0.2}}f_i(e)<2$$
for every pair $\{x,y\}\in\binom{V(H^3_r)}{2}$.

We use Lemma \ref{frac} to find fractional perfect matchings in $H_r^3$ as follows. Let $G_1:=H_r^3$.
  By the definition of $H_r^3$, one can see that $H^3_{n_1,r}([n],[n_1]\setminus [n])$ is a subgraph of $G_1$.  By Lemma \ref{frac}, $G_1$ has a fractional perfect matching $f_1$ satisfying $\sum_{\{x,y\}\subseteq e}f_1(e)\leq 1$ for every pair $\{x,y\}\in\binom{V(H_r^3)}{2}$.
	Suppose that we have found $t$ fractional perfect matchings $f_1,\ldots,f_t$ in $H^3_r$ such that $\sum_{\{x,y\}\subseteq e}\sum_{i=1}^{t}f_i(e)<2$ for every pair $\{x,y\}\in\binom{V(H^3_r)}{2}$. If $t\geq n^{0.2}$, then $f_1,\ldots,f_{n^{0.2}}$ are desired fractional perfect matchings. So we may assume that $1\leq t<n^{0.2}$.   Let $$S_t:=\left\{\{x,y\}\in \binom{[n_1]}{2}:\sum_{\{x,y\}\subseteq e}\sum_{i=1}^t f_i(e)\geq 1\right\}.$$  Let  $E_t:=\{e\in E(H_r^3):\binom{e}{2}\cap S_t\neq \emptyset\}$ and let $G_{t+1}:=H^3_r-E_{t}$. Note that $\Delta_2(G_1)\leq n_1-2$. So we have
	\begin{align*}
		|E_t|&\leq |S_t| (n_1-2)\\
        &<  n_1 \sum_{\{x,y\}\in S_t}\left(\sum_{\{x,y\}\subseteq e}\sum_{i=1}^tf_i(e)\right)\\
	    &\leq n_1\sum_{\{x,y\}\in\binom{[n_1]}{2}}\left(\sum_{\{x,y\}\subseteq e}\sum_{i=1}^t f_i(e)\right)\\
        &=n_1 \sum_{i=1}^t\left(\sum_{\{x,y\}\in\binom{[n_1]}{2}}\sum_{\{x,y\}\subseteq e} f_i(e)\right)\\
        &=n_1 \sum_{i=1}^t\left(3\sum_{e\in E(H^3_r)} f_i(e)\right)\\
		&= tn_1^2
		\leq n_1^{2.2}.
	\end{align*}
 We claim that	 for every $x\in V(H^3_r)$, there are at most $2t$ distinct vertices $y_1,\ldots,y_{2t}\in V(H^3_r)$ such that $\{x,y_i\}\in S_t$ for $1\leq i\leq 2t$. Otherwise, suppose that there are $2t+1$ distinct vertices, saying $y_1,\ldots,y_{2t+1}\in V(H^3_r)$ such that $\{x,y_j\}\in S_t$ for $1\leq j\leq 2t+1$.  Then $$2t=2\sum_{i=1}^{t}\sum_{x\in e}f_i(e)\geq \sum_{i=1}^t\sum_{j=1}^{2t+1}\sum_{\{x,y_j\}\subseteq e}f_i(e)\geq \sum_{j=1}^{2t+1}\left(\sum_{\{x,y_j\}\subseteq  e}\sum_{i=1}^tf_i(e)\right)\geq 2t+1,$$
a contradiction. Thus $d_{G_{t+1}}(x)\geq d_{H^3_r}(x)-2tn_1\geq d_{H^3_r}(x)-2n_1^{1.2}$ for every $x\in [n_1]$. Recall that $H^3_{n_1,r}([n],[n_1]\setminus [n])$ is a subgraph of $H^3_r$, one can see that
	$$|N_{H^3_{n_1,r}([n],[n_1]\setminus [n])}(x)\setminus N_{G_{t+1}}(x)|\leq 2n_1^{1.2}\leq \theta n_1^2$$ for every $x\in [n_1]$.
 So every vertex of $G_{t+1}$ is $\theta$-good with respect to $H^3_{n_1,r}([n],[n_1]\setminus [n])$.
Recall that $G_1-Q$  is not $\varepsilon$-close to  $H^3_{n,s}$ or $D^3_{n,s}$ and $E(G_{t+1}-Q)\subseteq E(G_1-Q)$. So we have  $G_{t+1}-Q$ is not $\varepsilon$-close to  $H^3_{n,s}$ or $D^3_{n,s}$. One can see that
\[
e(G_{t+1}-Q)\geq e(H)-|E_t|\geq \max\left\{\binom{n}{3}-\binom{n-s}{3},\binom{3s+2}{3}\right\}-2\rho n^3
 \]for sufficiently large $n$. By Lemma \ref{frac}, $G_{t+1}$ has a  factional perfect matching $f_{t+1}$.
Note that $f_{t+1}$ is a function defined on $E(H^3_r)\setminus E_t$. We extend $f_{t+1}$ to be a fractional perfect matching in $H^3_r$ by defining $f_{t+1}(e)=0$ for $e\in E_t$. Recall that $\sum_{\{x,y\}\subseteq e}\sum_{i=1}^{t}f_i(e)<1$ for every pair $\{x,y\}\notin S_t$ and $\sum_{\{x,y\}\subseteq e}\sum_{i=1}^{t}f_i(e)<2$ for every pair $\{x,y\}\in S_t$. Since $f_{t+1}(e)=0$ for $e\in E_{t}$, we have $\sum_{\{x,y\}\subseteq e}f_{t+1}(e)\leq1$ for every pair $\{x,y\}\notin S_t$ and $\sum_{\{x,y\}\subseteq e}f_{t+1}(e)=0$ for every pair $\{x,y\}\in S_t$. Thus
$\sum_{\{x,y\}\subseteq e}\sum_{i=1}^{t+1}f_i(e)<2$ for every pair $\{x,y\}\in {V(H_r^3)\choose 2}$.
Continuing the process, we may find  $n^{0.2}$ fractional perfect matchings $f_1,f_2,\ldots,f_{n^{0.2}}$ in $H^3_r$ such that $\sum_{\{x,y\}\subseteq e}\sum_{i=1}^{n^{0.2}}f_i(e)<2$ for every pair $\{x,y\}\in {V(H_r^3)\choose 2}$.

	
	Let $f_1,f_2,\ldots,f_{n^{0.2}}$ be $n^{0.2}$ fractional perfect matchings in $H^3_r$ such that for every $\{x,y\}\in\binom{V(H^3_r)}{2}$
	$$\sum_{\{x,y\}\subseteq e}\sum_{i=1}^{n^{0.2}}f_i(e)<2.$$
	Let $f:=\frac{1}{2}(f_1+\ldots+f_{n^{0.2}})$. One can see that $0\leq f(e)\leq 1$ for every $e\in E(H^3_r)$. We select a generalized binomial subgraph $H'$ of $H^3_r$ by independently choosing each edge $e\in E(H^3_r)$ with probability $f(e)$.
	
  Next we will show $d_{H'}(x)\thicksim n^{0.2}/2$ for any vertex $x$ and $\Delta_2(H')\leq n^{0.1}$.	Recall that $f_i$ is a fractional perfect matching, thus $\sum_{x\in e}f_i(e)=1$ for $1\leq i\leq n^{0.2}$. 
One can see that 	\begin{equation*}
		\mathbb{E}d_{H'}(x)=\sum_{x\in e} f(e)=\frac{1}{2}\sum_{i=1}^{n^{0.2}}\sum_{x\in e}f_i(e)= \frac{1}{2}n^{0.2}.
	\end{equation*}
	Hence by Chernoff's inequality (\ref{chers}),
	\begin{equation*}
		\mathbb{P}(|d_{H'}(x)-n^{0.2}/2|\geq \alpha n^{0.2}/2)\leq2e^{-\frac{\alpha^2}{6}n^{0.2}}.
	\end{equation*}
	Set $\alpha=n^{-0.05}$. We have  $|d_{H'}(x)-n^{0.2}/2|\leq n^{0.15}/2$ with probability $1-O(e^{-n^{0.1}})$. Thus for all $x\in V(H')$, we have $d_{H'}(x)= (\frac{1}{2}+o(1))n^{0.2}$ with probability $1-o(1)$.
	
	Moreover, for all pairs $x,y\in V(H')$, 
	\begin{equation*}
		\mathbb{E}d_{H'}(\{x,y\})=\sum_{\{x,y\}\subseteq e}f(e)=\frac{1}{2}\sum_{i=1}^{n^{0.2}}\sum_{\{x,y\}\subseteq e}f_i(e)\leq1.
	\end{equation*}
	 Hence,  by Chernoff's inequality (\ref{cherl}) for large deviations, when $n$ is sufficiently large,
	\begin{equation*}
		\mathbb{P}(d_{H'}(\{x,y\})\geq n^{0.1})\leq e^{-\Omega(n^{0.1})}.
	\end{equation*}
	So for all pairs of vertices $x,y\in V(H')$, $d_{H'}(\{x,y\})\leq n^{0.1}$ with  probability $1-o(1)$.
	
	Thus by Theorem~\ref{deg}, $H_r^3$ has a matching covering all but at most $\sigma n$ vertices, where $\sigma\ll\eta$ is a positive constant. Hence we have $\nu(H_r^3)\geq (n+r-\sigma n)/3>s+r$. This completes the proof.
\qed

%

\section{Proof of Theorem \ref{main}}
Note that $\binom{n}{3}-\binom{n-s}{3}=\binom{n}{3}-\binom{n-s}{3}-\binom{n-s-3}{2}+1+o(n^3)$. By inequality (\ref{diff}), one can see that $\binom{n}{3}-\binom{n-s}{3}\leq \binom{3s+2}{3}$ for sufficiently large $n$ and $s> \frac{13}{45}n$. Thus by Theorems \ref{erdos3} and \ref{bollobas}, we may assume that $\frac{n}{54}\leq s\leq \frac{13}{45}n$.

Lemmas \ref{closecase} and \ref{nonclose} imply that $e(H)\leq\max\{\binom{n}{3}-\binom{n-s}{3}-\binom{n-s-3}{2}+1,\binom{3s+2}{3}\}$ for sufficiently large $n$ and $\frac{n}{54}\leq s\leq \frac{13}{45}n$.          \qed

\bigskip

{\sc Acknowledgement}. The authors would like to thank Jie Han for sharing their work.

\end{document}